\documentclass[10pt]{amsart}
\usepackage{amssymb,amsmath}
\usepackage{graphicx}
\usepackage{mathrsfs}  
\usepackage{youngtab}
\usepackage{young}
\allowdisplaybreaks
\usepackage[usenames,dvipsnames]{xcolor}

\newcommand{\floor}[1]{\lfloor #1 \rfloor}
\def\beq*{\begin{eqnarray*}}
\def\eeq*{\end{eqnarray*}}
\def\bqn{\begin{eqnarray}}
\def\eqn{\end{eqnarray}}
\def\be{\begin{equation}}
\def\ee{\end{equation}}
\def\lbl{\label}

\def\b{\beta}

\def\a{\alpha}
\def\d{\delta}

\def\deg{\operatorname{deg}}

\def\var{{{\mathbb V}\mbox{ar}}}

\def\mcT{\mathcal{T}}

\def\mbR{\mathbb{R}}
\def\mbP{\mathbb{P}}
\def\mbE{\mathbb{E}}

\def\mcX{\mathcal{X}}
\def\mcP{\mathcal{P}}

\newcommand{\ris}[1]{\overline{#1}}	

\newtheorem{theorem}{Theorem}[section]

\newtheorem{proposition}[theorem]{Proposition}
\newtheorem{definition}[theorem]{Definition}

\numberwithin{subcase}{case}

\author[Pawe{\l} Hitczenko]{Pawe{\l} Hitczenko}
\address{Department of Mathematics, Drexel University, Philadelphia, 
PA  19104, USA} 
\email{phitczenko@math.drexel.edu}

\author[Aleksandr Yaroslavskiy]{Aleksandr Yaroslavskiy
}
\address{Department of Mathematics, Drexel University, Philadelphia, 
PA  19104, USA} 
\email{amy46@drexel.edu}

\keywords{Tree--like tableaux, permutation tableaux,  partially asymmetric simple exclusion process}
\subjclass[2010]{05A05, 05A15, 60C05, 60F05}

 \title[Distribution of the Number of Corners in Tree-Like Tableaux]{Distribution of the Number of Corners in Tree--like Tableaux}

\begin{document}

\begin{abstract}
 
   In this paper, we study tree--like tableaux and some of their probabilistic properties. Tree--like tableaux are in bijection with other combinatorial structures, including permutation tableaux,  and have a connection to the partially asymmetric simple exclusion process (PASEP),  an important  model of an interacting particles system. In particular, in the context of tree-like tableaux, a corner corresponds to a node occupied by a particle that could jump to the right while inner corners indicate a particle with an empty node to its left. Thus, the total number of corners represents the number of nodes at which PASEP can move, i. e. the total  current activity of the system. As the number of inner corners and regular corners is connected, we limit our discussion to just regular corners and show that asymptotically, the number of corners in a tableau of length $n$ is normally distributed.
\end{abstract}


\maketitle

\section{Introduction}
In this report, we study tree-like tableaux, a combinatorial object introduced in \cite{abn}. They are in bijection with permutation tableaux and alternative tableaux but are interesting in their own right as they exhibit a natural tree structure. Aside from being in bijection with permutations and permutation tableaux, they can be used to study the partially asymmetric simple exclusion process (PASEP). The PASEP (see e. g. \cite{cw,dehp} and references therein) is a model in which $n$ nodes on a 1-dimensional lattice each either contain a particle or not. At each time interval, a particle can either move left or right to an empty adjacent node with fixed probabilities  and the probability of a move left is $q$ times the probability of jumping to the right. New particle may also enter from the left with probability $\a$ (if the first node is unoccupied) and a particle on the $n$th node may leave the lattice with probability $\b$.  A state of the PASEP is a configuration of occupied and unoccupied nodes and it naturally corresponds to border edges of tree--like tableaux.
In this association,  corners in tree--like tableaux correspond to sites at which a particle can move  (we will give more details below, see also \cite{lz} for an explanation). In physics literature this is known as (total) current activity \cite{ds04,ds05} and  was studied for the TASEP (a special case of the PASEP with $q=0$) in \cite{sq}.

 It was conjectured (see  \cite[Conjecture~4.1]{lz}) that the expected number of corners in a randomly chosen tree--like tableaux of size $n$ is $(n+4)/6$.
 This conjecture (and its companion for symmetric tree--like tableaux) was  proved in \cite[Theorem~4]{hl} and subsequently also in \cite[Theorem~4.1]{ggls}). However, not much beyond that has been known (even the asymptotic value of the variance).  In the present paper we take the next step in the analysis of tree--like tableaux. First, we obtain the variance of the number of corners. Furthermore,  we   also show that the number of corners in random tree--like tableau of size $n$ is asymptotically normal as $n$ goes to infinity.

The rest of the paper is organized as follows. In the next section we introduce  the necessary definitions and  notation.
  We also explain the relation between the tree--like tableaux and the PASEP and state our main result on the distribution of the number of corners. 
 In Section~\ref{sec:gf}   we present a recursive relation for the generating function involving the corners in a similar combinatorial object, namely permutation tableaux. This recursion will be used in Section~\ref{sec:mgf} to obtain a recursion for the moment generation function of the number of corners in permutation tableaux and in Section~\ref{sec:clt} to conclude the proof of asymptotic normality. Since the number of corners in two types of the tableaux are closely related (and the difference is asymptotically negligible after  normalization) this will  immediately imply the same result for the number of corners in tree--like tableaux.

\section{Preliminaries and statement of main result}
\subsection{Tree--like Tableaux and Permutation Tableaux}
We endeavor to introduce the background for studying tree--like tableaux. We start by recalling the necessary notions and properties.  
\begin{definition}
A Ferrers diagram is an up and left justified arrangement of cells with weakly decreasing number of cells in rows. Depending on the situation, some rows may or may not be empty. The length of a Ferrers diagram is the number of columns plus the number of rows. 
 \end{definition}

Let us recall the following  definition introduced in \cite{abn}.

 \begin{definition}\label{T}
A tree--like tableau of size $n$ is a Ferrers diagram of length $n+1$ with no empty rows and with some cells (called pointed cells) filled with a point according to the following rules:
\begin{enumerate}
\item The cell in the first column and first row is always pointed (this point is known as the root point). \label{T1}
\item Every row and every column contains at least one pointed cell. \label{T2}
\item For every non--root pointed cell, either all the cells above are empty or all the cells to the left are empty (but not both). \label{T3}
\end{enumerate}
We denote the set of all tree--like tableaux of size $n$ by $\mcT_n$.
\end{definition}

\begin{figure}[h]
\setlength{\unitlength}{.5cm}
\begin{center}
\begin{picture} (8,8)

\put(0,0){\line(0,1){7}}
\put(1,0){\line(0,1){7}}
\put(2,1){\line(0,1){6}}
\put(3,3){\line(0,1){4}}
\put(4,3){\line(0,1){4}}
\put(5,3){\line(0,1){4}}
\put(6,5){\line(0,1){2}}
\put(7,5){\line(0,1){2}}

\put(0,7){\line(1,0){7}}
\put(0,6){\line(1,0){7}}
\put(0,5){\line(1,0){7}}
\put(0,4){\line(1,0){5}}
\put(0,3){\line(1,0){5}}
\put(0,2){\line(1,0){2}}
\put(0,1){\line(1,0){2}}
\put(0,0){\line(1,0){1}}

\put(4.3, 3.3){$\bullet$}
\put(5.3, 5.3){$\bullet$}
\put(6.3, 6.3){$\bullet$}
\put(3.3, 6.3){$\bullet$}
\put(2.3, 3.3){$\bullet$}
\put(1.3, 2.3){$\bullet$}
\put(1.3, 4.3){$\bullet$}
\put(1.3, 5.3){$\bullet$}
\put(1.3, 6.3){$\bullet$}
\put(0.3, 3.3){$\bullet$}
\put(0.3, 1.3){$\bullet$}
\put(0.3, 0.3){$\bullet$}
\put(0.3, 6.3){$\bullet$}

\end{picture}

\end{center}

\caption{A tree--like tableau of size $13$.} \lbl{fig:tlt} 

\end{figure}
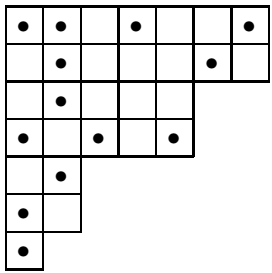

We will also need a notion of permutation tableaux originally introduced in \cite{sw}. 
\begin{definition}\label{P}
A permutation tableau of size $n$ is a Ferrers diagram of length $n$ whose non--empty rows are filled with $0$'s and $1$'s according to the following rules: 
\begin{enumerate}
\item Each column has at least one 1. \label{P1}
\item Any 0 cannot have a 1 both above it and to the left of it simultaneously. \label{P2} 
\end{enumerate}
We denote the set of all permutation tableaux of size $n$ by $\mcP_n$.
\end{definition}

\begin{figure}[h]
\setlength{\unitlength}{.5cm}
\begin{center}
\begin{picture} (24,7)
%
\put(0,3){\line(0,1){4}}
\put(1,3){\line(0,1){4}}
\put(2,4){\line(0,1){3}}
\put(3,5){\line(0,1){2}}
\put(4,6){\line(0,1){1}}
\put(5,6){\line(0,1){1}}
\put(6,6){\line(0,1){1}}
\put(7,6){\line(0,1){1}}

\put(0,7){\line(1,0){7}}
\put(0,6){\line(1,0){7}}
\put(0,5){\line(1,0){3}}
\put(0,4){\line(1,0){2}}
\put(0,3){\line(1,0){1}}

\put(4.3, 6.3){$1$}
\put(5.3, 6.3){$1$}
\put(6.3, 6.3){$1$}
\put(3.3, 6.3){$1$}
\put(2.3, 6.3){$0$}
\put(1.3, 6.3){$1$}
\put(2.3, 5.3){$1$}
\put(1.3, 5.3){$0$}
\put(1.3, 4.3){$1$}
\put(0.3, 6.3){$0$}
\put(0.3, 4.3){$1$}
\put(0.3, 3.3){$0$}
\put(0.3, 5.3){$0$}


\put(10,1){\line(0,1){6}}
\put(11,3){\line(0,1){4}}
\put(12,4){\line(0,1){3}}
\put(13,4){\line(0,1){3}}
\put(14,4){\line(0,1){3}}
\put(15,6){\line(0,1){1}}
\put(16,6){\line(0,1){1}}

\put(10,7){\line(1,0){6}}
\put(10,6){\line(1,0){6}}
\put(10,5){\line(1,0){4}}
\put(10,4){\line(1,0){4}}
\put(10,3){\line(1,0){1}}


\put(10.25, 3.25){$1$}
\put(10.25, 4.25){$0$}
\put(10.25, 5.25){$0$}
\put(10.25, 6.25){$0$}

\put(11.25, 6.25){$1$}
\put(11.25, 5.25){$0$}
\put(11.25, 4.25){$1$}

\put(12.25, 6.25){$0$}
\put(12.25, 5.25){$1$}
\put(12.25, 4.25){$1$}

\put(13.25, 6.25){$0$}
\put(13.25, 5.25){$1$}
\put(13.25, 4.25){$1$}

\put(14.25, 6.25){$1$}

\put(15.25, 6.25){$1$}


\put(20,3){\line(0,1){4}}
\put(21,3){\line(0,1){4}}
\put(22,3){\line(0,1){4}}
\put(23,4){\line(0,1){3}}

\put(20,7){\line(1,0){3}}
\put(20,6){\line(1,0){3}}
\put(20,5){\line(1,0){3}}
\put(20,4){\line(1,0){3}}
\put(20,3){\line(1,0){2}}

\put(20.25, 6.25){$1$}

\put(20.25, 5.25){$0$}
\put(21.25, 5.25){$0$}

\put(21.25, 4.25){$1$}
\put(20.25, 4.25){$0$}
\put(22.25, 4.25){$1$}

\put(20.25, 3.25){$0$}
\put(21.25, 3.25){$1$}

\put(21.25, 6.25){$0$}
\put(22.25, 6.25){$1$}
\put(22.25, 5.25){$0$}

\end{picture}
\end{center}
\caption{Examples of permutation tableaux. The tableau in the middle has  two empty rows.} \lbl{fig:pt}
\end{figure}
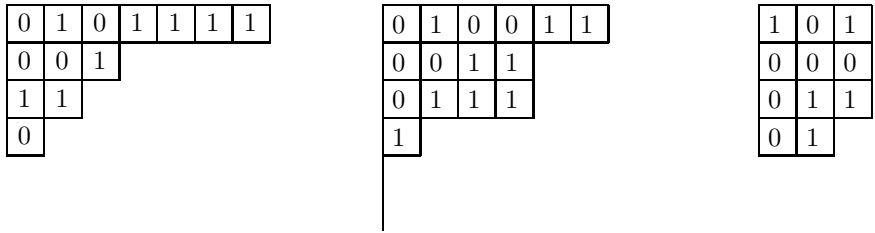
In a tree--like or  a permutation tableau,  the edges  outlining the southeast border are often called border edges. We also refer to those edges as steps. Each step is either a south step or a west step if we move along  border edges from northeast to southwest or a north step or an east step if we move in the opposite direction.
\begin{definition}
A corner in a tableau is a south step followed immediately by a west step as we traverse the border edges starting from the  northeast and going to the southwest 
 end. We denote by $c(T)$ the number of corners of the tableau $T$. If $\mcT$ is a set of tableaux we let 
\[c(\mcT)=\sum_{T\in\mcT}c(T)\]
 denote the total number of corners of tableaux in $\mcT$. 
\end{definition}
Tree--like tableaux correspond to  the states of the PASEP as follows:  traverse the border edges of a tree--like tableau beginning at the southwest end. Ignoring the first and the last step, a north step corresponds to an unoccupied node and an east step corresponds to an occupied node. Thus, for example, the tree--like tableau depicted in Figure~\ref{fig:tlt} corresponds to the following state of the PASEP on $12$ nodes:  

\begin{figure}[htbp]
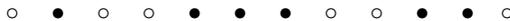
 
\begin{center}
\[
\circ\quad\bullet\quad\circ\quad\circ\quad\bullet\quad\bullet\quad\bullet\quad\circ\quad\circ\quad\bullet\quad\bullet\quad\circ\]
\end{center}
\caption{The  state of the PASEP corresponding to the tree--like tableau in Figure~\ref{fig:tlt}.} \lbl{fig:pasep}
\end{figure}
(In this state of the PASEP a particle could enter from the left, the particle in the second node could jump in either direction, the particle in the fifth or the tenth node could jump to the left and a particle in the seventh or the eleventh node could jump to the right.)  

With this association, the corners in tree--like tableau correspond to occupied sites, in which the particle could jump to the right (or enter from the left, or leave to the right) and any inner corner (north step followed by the east step) corresponds to an occupied node  with a particle that can jump to the left. Since the number of inner corners is one less than the number of corners, the total number of possible moves for the PASEP in a state corresponding to $T\in\mcT_n$ is  $2c(T)-1$. For example the tableau in Figure~\ref{fig:tlt} has four corners and thus the PASEP in the state depicted in Figure~\ref{fig:pasep}
has seven possible moves as described above. As we mentioned earlier, in physics literature the number of nodes at which a particle can move is called the current activity of the system, see e.~g.  \cite{ds04,ds05,sq}.

It is known (see \cite[Proposition~3.1]{abn}) that  tree--like tableaux of length $n+1$ are in bijection with  permutation tableaux of length $n$ (and both are in bijection with permutations of $[n]$, see e.~g.~\cite{bur,cn,sw,abn}). The corners  need not be preserved, but a difference between their number in a tableau and its image under that bijection is at most one (see \cite[Section~3]{hl}).  
Therefore,  in order to study corners in tree--like tableaux it will be enough to study corners in permutation tableaux, and this is what we are going to do. We need a few more notions associated with permutation tableaux.

\begin{definition}
We say a zero in a permutation tableau is restricted if it has a one above it. Otherwise, the zero is unrestricted. We say a row is restricted if it contains a restricted zero, otherwise it is unrestricted. We denote by $u(T)$ the number of unrestricted rows of $T$.
\end{definition}
In the first example in Figure~\ref{fig:pt}, the top  and the third row are unrestricted, but the other two rows are restricted.  Note that the top row of a permutation tableau is necessarily  unrestricted. 

An important feature of permutation tableaux is that they can be constructed recursively. Given a permutation tableau, we can increase its length incrementally and fill in the new columns as they come.
\begin{definition}
We say a tableau $T\rq{}\in \mathcal{P}_{n+1}$ is an extension of a tableau $T\in \mathcal{P}_{n}$ if $T\rq{}$ is obtained either by adding a south step to $T$ or by adding a west step to $T$ and filling the new  column according to the rules. 
\end{definition}
Notice that there is only one way to extend a tableau by adding a south step, but multiple ways by adding a west step. When a west step is added, a new column is formed which must be filled. In a cell that is part of a restricted row, it must have a zero. The cells that are part of the unrestricted rows leave us options. It is not difficult to count the number of extensions    (see e.~g.~\cite{ch,hj}) and we have: 
\begin{proposition}
The number of extensions of $T\in \mathcal{P}_{n}$ into $T\rq{}\in \mathcal{P}_{n+1}$ is $2^{u(T)}$. 
\end{proposition}
This, however, tells us nothing of the number of unrestricted rows of the extended tableau, which is often of relevance. But the evolution of the number of unrestricted rows can be traced down (see \cite{ch} or \cite{hj}) and is given by:

\begin{proposition}\lbl{prop:ext}
Let $T\in \mathcal{P}_{n}$ be a permutation tableau of length $n$, and let $u(T)$ be the number of unrestricted rows of $T$. The number of ways to extend $T$ so that the extension has exactly $k$
 unrestricted rows, $1\le k\le u(T)$,  is:
\[\sum_{j=1}^{k} \binom {u(T)-j}{k-j}=\binom{u(T)}{k-1}.\]
\end{proposition}

In the following sections we prefer to use  probabilistic language and thus, instead of talking about the number of corners in tableaux,  we let $\mbP_n$  be the uniform probability measure 
on $\mcX_n$ (where $\mcX_n$ is either $\mcT_n$ or $\mcP_n$)
 and consider a random variable $C_n$ on the probability space $(\mcX_n,\mbP_n)$ defined by $C_n(T)=c(T)$, the number of corners of $T\in\mcX_n$.
A tableau chosen from $\mcX_n$ according to the probability measure $\mbP_n$ is usually referred to as a random tableau of size $n$ and $C_n$ is referred to as the number of corners in a random tableau of size $n$. We let $\mbE_n$ denote the expected value with respect to the measure $\mbP_n$. 
Then, of course, we have: 
\[\mbE_nC_n=\frac{c(\mcX_n)}{|\mcX_n|}
.\]
As we will see below,  
the variance of the number of corners, $\var(C_n)$, grows to infinity as $n\to\infty$ (in fact, $\var(C_n)\sim11n/180$).  Furthermore if $\phi_n:\ \mcT_n\to \mcP_n$  is the bijection described in \cite{abn,hl} then for $T\in\mcT_n$, $c(T)=c(\phi_n(T))+I$, where $I$ is 0 or 1 depending on the shape of $T$. Therefore, for every $x\in\mbR$
\[\mbP_n\left(T\in\mcT_n:\ \frac{C_n(T)-\mbE C_n}{\sqrt{\var(C_n)}} \le x\right)=\mbP_n\left(T\in\mcP_n:\  \frac{C_n(T)-\mbE C_n+O(1)}{\sqrt{\var(C_n)}}\le x\right).\]
Thus, the limiting distribution of  the number of corners in a random tree--like tableau is the same as that of the number of corners in a random permutation tableau, so we will focus on the latter. 
Our main result is as follows.
\begin{theorem} \lbl{main_thm}
Let $\{C_n\}$ be a sequence of random variables where $C_n$ is the number of corners in a random permutation tableau of length $n$. Let:
\[\mu_n=\frac{n+4}6-\frac1{n}\sim\frac n 6\]
and
$$\sigma_n^2=\var(C_n)\sim
\frac {11}{180}n.$$
Then 
\[\frac{C_n-\mu_n}{\sigma_n}\xrightarrow[]{d}\mathcal{N}(0,1) \quad or\quad\frac{C_n-\frac n 6}{\sqrt{\frac {11}{180}n}}\xrightarrow[]{d}\mathcal{N}(0,1), \]
where $\xrightarrow[]{d}$ is convergence in distribution and $\mathcal{N}(0,1)$ is the standard normal random variable.
\end{theorem}

\section{Generating Function and the First Two Moments}\lbl{sec:gf}
We wish to construct a generating function for the  number of corners  in  permutation tableaux of length $n$. We can do it recursively by using the extension procedure for permutation tableaux mentioned earlier. In order to do this we need to keep track of the number of unrestricted rows, and we use it as 
a 'catalytic' variable. Proposition~\ref{prop:ext} allows us to follow the evolution of the number of unrestricted rows under the extension and with its help we can derive a recurrence for the bivariate generating function. 
\begin{proposition}\label{prop:gf_rec}
Let  for $n\ge0$ 
$$C_n(x,z)=\sum_{T\in \mathcal{P}_n} x^{c(T)}z^{u(T)}$$ be the bivariate generation function of permutation tableaux of length $n$, where $x$ marks the number of corners and $z$ marks the number of unrestricted rows. Then we have the following recurrence for $C_n(x,z)$:
\be\lbl{gf_rec}C_{n}(x,z)=
zC_{n-1}(x,z+1)+(x-1)\Big(z(z+1)C_{n-2}(x,z+1)-z^2C_{n-2}(x,z)\Big)\ee
with 
\[C_0(x,z)=1,\quad C_1(x,z)=z.\]
\end{proposition}
\begin{proof} The initial condition is clear as there are no permutation tableaux of length $0$ and there is one permutation tableau of length $1$; it has one unrestricted row and no corners.  
To establish \eqref{gf_rec}, we split the set $\mcP_n$ according to  the number of moves west since the last corner of the tableau $T$. Specifically, 
 let $\mathcal{P}_{n,j}$ be the set of permutation tableaux of length $n$ obtained from tableaux of length $n-j-1$ by  
adding a move south, followed by $j$ moves west.
Then,
\bqn \nonumber C_n(x,z)&=&
\lbl{cor_split}\sum_{j=0}^{n-1}\sum_{T\in \mathcal{P}_{n,j}} x^{c(T)}z^{u(T)}\\&=&\sum_{T\in \mathcal{P}_{n,0}} x^{c(T)}z^{u(T)}
+\sum_{j=1}^{n-1}\sum_{T\in \mathcal{P}_{n,j}} x^{c(T)}z^{u(T)}.
\eqn
Note that $\mcP_{n,0}$ consists of tableaux in $\mcP_{n-1}$ extended by a south move (this does not change the number of corners, but increases the number of unrestricted rows by 1). Thus,  
\be\lbl{1st}\sum_{T\in \mathcal{P}_{n,0}} x^{c(T)}z^{u(T)}=z\sum_{S\in \mathcal{P}_{n-1}} x^{c(S)}z^{u(S)}=zC_{n-1}(x,z).
\ee
Next, recall from Proposition~\ref{prop:ext}, that if we add a column to the tableau $T$,  the number of ways to fill it so that the new tableau 
has $U$ unrestricted rows is: 
\[\sum_{k=1}^U \binom{u(T) - k}{U-k}.\]
Moreover, a tableau in $\mcP_{n,1}$ is obtained by first extending a tableau in $\mcP_{n-2}$ by a south step (this increases the number of unrestricted rows by 1) and then adding a step west and filling the column. Thus, using  the  identity
\[\sum_{k=1}^n \binom{m-k}{n-k}=\binom{m}{n-1}\]
in the third step below and the binomial formula in the fifth we obtain
\bqn\nonumber\sum_{j=1}^{n-1}\sum_{T\in \mathcal{P}_{n,j}} x^{c(T)} z^{u(T)}&=&\sum_{T\in \mathcal{P}_{n,1}} x^{c(T)}z^{u(T)}+\sum_{j=2}^{n-1}\sum_{T\in \mathcal{P}_{n,j}} x^{c(T)}z^{u(T)}\\&=&
\nonumber x\sum_{S\in \mathcal{P}_{n-2}} x^{c(S)}\sum_{u=1}^{u(S)+1}z^{u}\sum_{k=1}^u\binom{u(S)+1-k}{u-k}\\&&\qquad+
\nonumber\sum_{j=2}^{n-1}\sum_{S\in \mathcal{P}_{n-1,j-1}} x^{c(S)}\sum_{u=1}^{u(S)}z^{u}\sum_{k=1}^u\binom{u(S)-k}{u-k}\\&
\nonumber=&x\sum_{S\in \mathcal{P}_{n-2}} x^{c(S)}\sum_{u=1}^{u(S)+1}z^{u}\binom{u(S)+1}{u-1}\\&&\qquad+
\nonumber\sum_{j=2}^{n-1}\sum_{S\in \mathcal{P}_{n-1,j-1}} x^{c(S)}\sum_{u=1}^{u(S)}z^{u}\binom{u(S)}{u-1}\\&=&
\nonumber xz\sum_{S\in \mathcal{P}_{n-2}} x^{c(S)}\sum_{u=0}^{u(S)}z^{u}\binom{u(S)+1}{u}\\&&\qquad+
\nonumber z\sum_{j=2}^{n-1}\sum_{S\in \mathcal{P}_{n-1,j-1}} x^{c(S)}\sum_{u=0}^{u(S)-1}z^{u}\binom{u(S)}{u}\\&=&
\nonumber xz\sum_{S\in \mathcal{P}_{n-2}} x^{c(S)}[(1+z)^{u(S)+1}-z^{u(S)+1}]\\&&\qquad+
\nonumber z\sum_{j=2}^{n-1}\sum_{S\in \mathcal{P}_{n-1,j-1}} x^{c(S)}[(1+z)^{u(S)}-z^{u(S)}]\\&=&
\lbl{cn-2}xz(z+1)C_{n-2}(x,z+1)-xz^2C_{n-2}(x,z)\\&&\qquad+
\lbl{sum_cn-1}z\sum_{j=1}^{n-2}\sum_{S\in \mathcal{P}_{n-1,j}} x^{c(S)}[(1+z)^{u(S)}-z^{u(S)}].
\eqn
The sum \eqref{sum_cn-1} is
\beq*
&&
z\sum_{j=0}^{n-2}\sum_{S\in \mathcal{P}_{n-1,j}} x^{c(S)}[(1+z)^{u(S)}-z^{u(S)}]-z\sum_{S\in \mathcal{P}_{n-1,0}} x^{c(S)}[(1+z)^{u(S)}-z^{u(S)}]\\&&\quad=
z\Big[C_{n-1}(x,z+1)-C_{n-1}(x,z)\Big]
-
z\Big[(z+1)C_{n-2}(x,z+1)-zC_{n-2}(x,z)\Big]
\eeq*
where in the last line we used \eqref{1st}. 
Substituting this in \eqref{sum_cn-1} and combining  with \eqref{cn-2}, \eqref{cor_split}, \eqref{1st} and simplifying we arrive at: 
\beq*C_{n}(x,z)&=&zC_{n-1}(x,z)+xz(z+1)C_{n-2}(x,z+1)-xz^2C_{n-2}(x,z)\\&&\qquad+zC_{n-1}(x,1+z)
-zC_{n-1}(x,z)\\&&\qquad-z(1+z)C_{n-2}(x,1+z)+z^2C_{n-2}(x,z)
\eeq*
which is equivalent to \eqref{gf_rec}.
\end{proof} 
\subsection{Expectation}
The above proposition allows us to recover the expected value of the number of corners, a result conjectured in \cite{lz}, first proved in \cite{hl}, and then also in \cite{ggls}. To do this,  note that it  is clear from  \eqref{gf_rec} that 
\[C_{n}(1,z)=
zC_{n-1}(1,z+1)=\cdots=z^{\ris n},\]
where
\[z^{\ris n}=z(z+1)\cdot\dots\cdot(z+n-1),\]
is the rising factorial. We can treat 
\[\frac{C_n(x,z)}{C_n(1,z)}  =\frac{C_n(x,z)}{z^{\ris{n}}}\]  
as the probability generating function of a random variable that depends on a parameter $z$ and, in fact, is defined on a probability space that depends on $z$. Ultimately, we will be interested in $z=1$ but it is convenient to proceed with more generality.  

When we write $C_n(x,z)$ in the form 
\[C_{n}(x,z)=\sum_{m=0}^{\lfloor{n/2}\rfloor}c_{n,m}(z)(x-1)^m,\]
then    the expected value of such random variable is $c_{n,1}(z)/z^{\ris n}$.
 Note that \eqref{gf_rec} yields
\[c_{n,m}(z)=zc_{n-1,m}(z+1)+z(z+1)c_{n-2,m-1}(z+1)-z^2c_{n-2,m-1}(z),\]
with the initial conditions $c_{n,0}=z^{\ris n}$, $n\ge0$.
 Iteration 
 gives
\begin{eqnarray}\nonumber c_{n,m}(z)&=&z(z+1)c_{n-2,m}(z+2)\\&&\nonumber+z(z+1)\Big((z+2)c_{n-3,m}(z+2)-(z+1)c_{n-2,m}(z+1)\Big)\\&=&
\nonumber z^{\ris k}c_{n-k,m}(z+k)\\&&\nonumber+\sum_{j=1}^kz^{\ris{j}}\Big((z+j)c_{n-j-1,m-1}(z+j)-(z+j-1)c_{n-j-1,m-1}(z+j-1)\Big)
\\&=&\lbl{cnm_expand}
z^{\ris{n-2m}}c_{2m,m}(z+n-2m)
\\&&\nonumber+\sum_{j=1}^{n-2m}z^{\ris{j}}\Big((z+j)c_{n-j-1,m-1}(z+j)-(z+j-1)c_{n-j-1,m-1}(z+j-1)\Big)
.\end{eqnarray}
When $m=1$ this becomes
\beq* c_{n,1}(z)&=&z^{\ris{n-2}}c_{2,1}(z+n-2)+\\&&\sum_{j=1}^{n-2}z^{\ris{j}}\Big((z+j)c_{n-j-1,0}(z+j)-(z+j-1)c_{n-j-1,0}(z+j-1)\Big)
\\&=& z^{\ris{n-2}}(z+n-2)+\\&&\sum_{j=1}^{n-2}z^{\ris{j}}\Big((z+j)(z+j)^{\ris{n-j-1}}-(z+j-1)(z+j-1)^{\ris{n-j-1}}\Big)\\&&\quad=
z^{\ris{n-1}}+
z^{\ris{n-2}}\sum_{j=1}^{n-2}\Big((z+j)(z+n-2)-(z+j-1)^2\Big)\\&&\quad=
z^{\ris{n-1}}+z^{\ris{n-2}}(n-2)(z+n-2)+\\&&\quad z^{\ris{n-2}}\sum_{j=1}^{n-2}\Big((z+j-1)(z+n-2)-(z+j-1)^2\Big)\\&&\quad=
z^{\ris{n-2}}\Big((n-1)(z+n-2)+\sum_{j=1}^{n-2}(z+j-1)(n-j-1)\Big)\\&&\quad=
z^{\ris{n-2}}(n-1)\Big((z+n-2)+\frac{(n-2)(n+3z-3)}6\Big)\\&&\quad=
z^{\ris{n-2}}(n-1)\frac{n^2+3zn+n-6}6.
\eeq*
Therefore,
\[\frac{c_{n,1}(z)}{z^{\ris n}}=\frac{(n-1)(n^2+3zn+n-6)}{6(z+n-1)_2}
\]
where $(w)_k=w(w-1)\dots (w-(k-1))$ is the falling factorial. 
When $z=1$ the above formula gives 
\[\mbE C_n=\frac{n^2+4n-6}{6n}=\frac{n+4}6-\frac1n\] which agrees with  \cite[Theorem~2]{hl}.

\subsection{Variance}
Calculation of the expected value  can be pushed further and we can obtain the variance of the number of corners, which has not been known before.
\begin{proposition}\label{prop:var} For $n\ge4$ we have
\[\var(C_n)=\frac{11n^4-191n^2+360n+180}{180n^2(n-1)}\sim\frac{11}{180}n\]
as $n\to\infty$. In addition
\[\var(C_1)=0,\quad\var(C_2)=\frac14,\quad\var(C_3)=\frac5{36}.\]
\end{proposition}
\begin{proof} Consider $n\ge4$ (the other three cases can be calculated directly). Our first goal  is to extract $c_{n,2}(z)$.
From \eqref{cnm_expand} used with $m=2$  we have 
\beq*c_{n,2}(z)&=&z^{\ris{n-4}}c_{4,2}(z+n-4)\\&&\quad+\sum_{j=1}^{n-4}z^{\ris j}\Big((z+j)c_{n-j-1,1}(z+j)-(z+j-1)c_{n-j-1,1}(z+j-1)\Big).\eeq*
Since 
\[c_{4,2}(z)=z(z+1)c_{2,1}(z+1)-z^2c_{2,1}(z)=z(z+1)^2-z^3=z(2z+1),\]
we see that 
\be\lbl{c42} z^{\ris{n-4}}c_{4,2}(z+n-4)=z^{\ris{n-3}}(2(z+n)-7).\ee
Furthermore,
\beq*&&z^{\ris j}(z+j)c_{n-j-1,1}(z+j)\\&&\quad=z^{\ris{n-3}}(z+j)
(n-j-2)\frac{(n-j-1)^2+3(z+j)(n-j-1)+n-j-7}6
\eeq*
and similarly,
\beq*&&z^{\ris j}(z+j-1)c_{n-j-1,1}(z+j-1)\\&&\quad=z^{\ris{n-4}}(z+j-1)^2
(n-j-2)\frac{(n-j-1)^2+3(z+j-1)(n-j-1)+n-j-7}6.
\eeq*Therefore, 
\beq*&&\sum_{j=1}^{n-4}z^{\ris j}\Big((z+j)c_{n-j-1,1}(z+j)-(z+j-1)c_{n-j-1,1}(z+j-1)\Big)
\\&&=
\frac{z^{\ris{n-4}}}6\sum_{j=1}^{n-4}(n-j-2)\Big\{(z+j)(z+n-4)\Big((n-j-1)^2+3(z+j)(n-j-1)+n-j-7\Big)\\&&\qquad-(z+j-1)^2\Big((n-j-1)^2+3(z+j-1)(n-j-1)+n-j-7\Big)\Big\}
\\&&=\frac{z^{\ris{n-4}}}{360}(n-4)(5n^5+30n^4z+45n^3z^2-19n^4-105n^3z-150n^2z^2-56n^3-120n^2z
\\&&\qquad+255nz^2+301n^2+825nz+90z^2-981n-2070z+2430.
\eeq*
When $z=1$ this equals
\[\frac{(n-4)!}{360}(n-3)(n-4)(5n^4+26n^3-38n^2-83n-150).
\]
Combining with \eqref{c42} we get
\beq*
c_{n,2}(1)&=&(n-3)!(2n-5)+\frac{(n-4)!}{360}(n-3)(n-4)(5n^4+26n^3-38n^2-83n-150)\\&=&\frac{(n-2)!}{360}(5n^4+16n^3-110n^2-151n+600).
\eeq*
The second factorial moment for the number of corners is thus given by:
\[
\mbE(C_n)_2=\mbE C_n(C_n-1)=\frac{2!}{n!}c_{n,2}(1)=\frac{5n^4+16n^3-110n^2-151n+600}{180n(n-1)}\]
and therefore,
\beq*
\var(C_n)&=&\mbE(C_n)_2-(\mbE C_n)^2+\mbE C_n\\&=&\frac{5n^4+16n^3-110n^2-151n+600}{180n(n-1)}-\left(\frac{n+4}6-\frac1{n}\right)^2+\left(\frac{n+4}6-\frac1{n}\right)
\\&
=&\frac{11n^4-191n^2+360n+180}{180n^2(n-1)}
\eeq*
as claimed.
\end{proof}

It is, however, increasingly difficult to find $c_{n,m}$ for higher $m$. Instead, we will use \eqref{gf_rec} to derive a recurrence for the moment generating function and rely on method of moments (see e.~g.~\cite[Theorem~30.2]{bil}) to establish the asymptotic normality of suitably normalized $(C_n)$.

\section{Moment Generating Function}\lbl{sec:mgf}
To derive the moment generating function  for the number of corners, we substitute  $x=e^t$ in the expression for $C(x,z)$. We will be interested in positive values of $z$,
 and to emphasize this we let $z=y >0$. Consider
\be\lbl{mgf}P_n(t,y):=e^{-\mu_n(y)t}\frac{C_n(e^t,y)}{y^{\ris n}},\quad P_0(t,y)=P_1(t,y)=1\ee
where 
\[\mu_0(y)=0;\quad\mu_n(y)=\frac{(n-1)(n^2+3yn+n-6)}{6(y+n-1)_2},\quad n\ge1.\]
(Notice that $\mu_n(1)$ is the expected value of $C_n$, the number of corners in permutation tableaux of size $n$.)  Then,  recurrence \eqref{gf_rec} translates into 
\beq*P_n(t,y)&=&e^{\a_n(y)t}P_{n-1}(t,y+1)\\&+&\frac{e^t-1}{(y+n-1)_2}\left((y+1)(y+n-2)e^{\b_n(y)t}P_{n-2}(t,y+1)-y^2e^{\d_n(y)t}P_{n-2}(t,y)\right),\eeq*
where
\bqn\lbl{alpha}\a_{n}(y)&=&\mu_{n-1}(y+1)-\mu_n(y)=-\frac{n+yn-y-2}{(y+n-1)_2},\\
\b_n(y)&=&\mu_{n-2}(y+1)-\mu_n(y),\nonumber
\\\d_n(y)&=&\mu_{n-2}(y)-\mu_n(y).\nonumber
\eqn
The explicit expressions for $\b_n(y)$ and $\d_n(y)$ are not important, what matters however is that each of these expressions is of order one as $n,y\to\infty$ and that this holds uniformly over $n+y\ge y_0$.
 In particular, there exist universal constants $C_i$, and $y_i$, $i=1,2,3$, such that for all $n\in\mathbb N$, $y>0$ such that $n+y\ge y_i$ 
 \be\lbl{coeff_bdds}|\alpha_{n}(y)|\le C_1,\quad |\beta_{n}(y)|\le C_2, \quad |\delta_{n}(y)|\le C_3.\ee 
For example,
\[|\a_{n}(y)|\le\frac{n+ny}{(y+n-1)_2}\le\frac{(n+y)+(n+y)^2}{(y+n-1)_2}\le2,\]
whenever $n+y\ge7$, and similar statements hold for $\b_n(y)$ and $\d_n(y)$.

We now derive a linear recurrence of the first order for
\[P_n^{(m)}(0,y)=\frac{\partial^mP_n(t,y)}{\partial t^m}\Big|_{t=0}.\] 
First
\beq*
&&P_n^{(m)}(t,y)=e^{\a_n(y)t}P_{n-1}^{(m)}(t,y+1)+
\sum_{k=0}^{m-1}{m\choose k}\a_{n}^{m-k}(y)e^{\a_n(y)t}P_{n-1}^{(k)}(t,y+1)\\&&\,
+\frac{e^t}{(y+n-1)_2}\sum_{k=0}^{m-1}{m\choose k}\Big[(y+1)(y+n-2)\sum_{i=0}^k{k\choose i}\b_n^{k-i}(y)e^{\b_n(y)t}P_{n-2}^{(i)}(t,y+1)\\&&\hspace{4cm}
-y^2\sum_{i=0}^k{k\choose i}\d_n^{k-i}(y)e^{\d_n(y)t}P_{n-2}^{(i)}(t,y)\Big]
\\&&\,
+\frac{e^t-1}{(y+n-1)_2}\Big((y+1)(y+n-2)e^{\b_n(y)t}P_{n-2}^{(i)}(t,y+1)-y^2e^{\d_n(y)t}P_{n-2}^{(i)}(t,y)\Big)^{(m)}.
\eeq*

At $t=0$ the last term vanishes and letting $P_n^{(m)}(y):=P_n^{(m)}(0,y)$ we get 
\begin{eqnarray}\nonumber&&P_n^{(m)}(y)=P_{n-1}^{(m)}(y+1)+\sum_{k=0}^{m-1}{m\choose k}\a_n^{m-k}(y)P_{n-1}^{(k)}(y+1)
\\&&\lbl{P_rec}+\frac1{(y+n-1)_2}\sum_{k=0}^{m-1}
\left\{(y+1)(y+n-2)P_{n-2}^{(k)}(y+1)\sum_{i=k}^{m-1}{m\choose i}{i\choose k}\b_n^{i-k}(y)\right.\\&&\left.\hspace{3.3cm}-
y^2P_{n-2}^{(k)}(y)\sum_{i=k}^{m-1}{m\choose i}{i\choose k}\d_n^{i-k}(y)\right\}.\nonumber
\end{eqnarray}
This recurrence is the starting point for establishing asymptotic normality for the number of corners in permutation tableaux. We present the detailed argument in the forthcoming section.

\section{Proof of Theorem~\ref{main_thm}}\lbl{sec:clt}
Our proof  relies on the method of moments (see e.~g. \cite[Theorem~30.2]{bil} and on the analysis of  recurrence \eqref{P_rec} for the moments which will allow us to establish that:
\be\lbl{mom} \frac{P_n^{(m)}(1)}{{(\frac{11} {180} n)}^{\frac m2}}\rightarrow
\begin{cases} 
      0, & m \text{ odd} \\
      \frac{m!}{2^{\frac m2} \cdot (m/2)!}, & m \text{ even.} \\
   \end{cases}
\ee

First, we isolate the two highest degree terms in recurrence \eqref{P_rec} (the remaining terms are  of lower order and thus do not contribute significantly, as we will demonstrate shortly). Then \eqref{P_rec} expands into 
\beq*&&P_n^{(m)}(y)=P_{n-1}^{(m)}(y+1)+{m\choose m-1}\a_n(y)P_{n-1}^{(m-1)}(y+1)\\&&\hspace{2cm}+{m\choose m-2}\a_n^2(y)P_{n-1}^{(m-2)}(y+1)
\\&&\hspace{2cm}+\frac1{(y+n-1)_2}
\left\{(y+1)(y+n-2)P_{n-2}^{(m-1)}(y+1){m\choose m-1}\right.\\&&\hspace{4cm}-y^2P_{n-2}^{(m-1)}(y){m\choose m-1}\\&&\quad+(y+1)(y+n-2)P_{n-2}^{(m-2)}(y+1)\left[{m\choose m-2}+{m\choose m-1}{m-1\choose m-2}\b_n(y)\right]\\&&\left.\hspace{3cm}-
y^2P_{n-2}^{(m-2)}(y)\left[{m\choose m-2}+{m\choose m-1}{m-1\choose m-2}\d_n(y)\right]\right\}\\&&\quad+\sum_{k=0}^{m-3}{m\choose k}\a_{n}^{m-k}(y)P_{n-1}^{(k)}(y+1)\\&&\quad
+\frac{1}{(y+n-1)_2}\sum_{k=0}^{m-3}{m\choose k}\Big[(y+1)(y+n-2)\sum_{i=0}^k{k\choose i}\b_n^{k-i}(y)P_{n-2}^{(i)}(y+1)\\&&\hspace{4cm}
-y^2\sum_{i=0}^k{k\choose i}\d_n^{k-i}(y)P_{n-2}^{(i)}(y)\Big].
\eeq*
Consider first the terms involving the factor ${m\choose m-1}$.
Write 
\beq*P_{n-2}^{(m-1)}(y+1)&=&P_{n-1}^{(m-1)}(y+1)+\Delta_1^{(m-1)}(n,y)
,\\ 
P_{n-2}^{(m-1)}(y)&=&P_{n-1}^{(m-1)}(y+1)+\Delta_2^{(m-1)}(n,y)
\eeq*
where we have set 
\beq*\Delta_1^{(k)}(n,y)&:=&P_{n-2}^{(k)}(y+1)-P_{n-1}^{(k)}(y+1),\\ \Delta_2^{(k)}(n,y)&:=&P_{n-2}^{(k)}(y)-P_{n-1}^{(k)}(y+1)=P_{n-2}^{(k)}(y)-P_{n-2}^{(k)}(y+1)+\Delta_1^{(k)}(n,y)
.\eeq*
Notice that 
\[\a_n(y)+\frac{(y+1)(y+n-2)}{(y+n-1)_2}-\frac{y^2}{(y+n-1)_2}=0\]
and therefore,
\beq* &&{m\choose m-1}\left\{\a_n(y)P_{n-1}^{(m-1)}(y+1)+\frac{(y+1)(y+n-2)}{(y+n-1)_2}
P_{n-2}^{(m-1)}(y+1)\right.\\&&\hspace{2.5cm}\left.-\frac{y^2}{(y+n-1)_2}P_{n-2}^{(m-1)}(y)\right\}\\&&=
m\left(\frac{(y+1)(y+n-2)}{(y+n-1)_2}\Delta_1^{(m-1)}(n,y)-\frac{y^2}{(y+n-1)_2}\Delta_2^{(m-1)}(n,y)\right)\eeq*
Similarly, the expression involving factor ${m\choose m-2}$  is
\beq*&&\left(P_{n-1}^{(m-2)}(y+1)\alpha_n^2(y)+P_{n-2}^{(m-2)}(y+1)\frac{(y+1)(y+n-2)}{(y+n-1)_2}
(1+2\beta_n(y))\right.
\\&&\qquad\left.
-P_{n-2}^{(m-2)}(y)\frac{y^2}{(y+n-1)_2}(1+2\delta_n(y))\right)=
\\&&P_{n-1}^{(m-2)}(y+1)\left\{\a_n^2(y)+\frac{(y+1)(y+n-2)(1+2\b_n(y))-y^2(1+2\d_n(y))}{(y+n-1)_2}\right\}
\\&&\quad+
\left(\Delta_1^{(m-2)}(n,y)\frac{(y+1)(y+n-2)}{(y+n-1)_2}
(1+2\beta_n(y))\right.
\\&&\qquad\left.
-\Delta_{2}^{(m-2)}(n,y)\frac{y^2}{(y+n-1)_2}(1+2\delta_n(y))\right).
\eeq*
We denote  the term in the braces by $T_n(y)$ and write it as 
\be\lbl{t_def}T_n(y)=\a_n^2(y)-\a_n(y)+2\frac{(y+1)(y+n-2)\b_n(y)-y^2\d_n(y)}{(y+n-1)_2},\quad n\ge2.\ee
It  follows from \eqref{coeff_bdds} that
 \be\lbl{t_bdd}|T_{n}(y)|\le C\ee 
 for an absolute constant $C$ and  all $n\in\mathbb N$, $y>0$, such that $n+y\ge y_0$.

With this notation,   recurrence \eqref{P_rec} simplifies further to
\be
\lbl{P_n}P_n^{(m)}(y)=P_{n-1}^{(m)}(y+1)+{m\choose 2}T_n(y)P_{n-1}^{(m-2)}(y+1)+R_m(n,y)
\ee
where 
\beq*
R_{m}(n,y)&=&m\left(\Delta_1^{(m-1)}(n,y)\frac{(y+1)(y+n-2)}{(y+n-1)_2}-\Delta_2^{(m-1)}(n,y)\frac{y^2}{(y+n-1)_2}\right)
\\&&\quad+{m\choose 2}
\left(\Delta_1^{(m-2)}(n,y)\frac{(y+1)(y+n-2)}{(y+n-1)_2}
(1+2\beta_n(y))\right.
\\&&\hspace{1.5cm}\left.
-\Delta_{2}^{(m-2)}(n,y)\frac{y^2}{(y+n-1)_2}(1+2\delta_n(y))\right).
\\&&+\sum_{k=0}^{m-3}{m\choose k}\a_{n}^{m-k}(y)P_{n-1}^{(k)}(y+1)\\&&\,
+\frac{1}{(y+n-1)_2}\sum_{k=0}^{m-3}{m\choose k}\Big[(y+1)(y+n-2)\sum_{i=0}^k{k\choose i}\b_n^{k-i}(y)P_{n-2}^{(i)}(y+1)\\&&\hspace{4cm}
-y^2\sum_{i=0}^k{k\choose i}\d_n^{k-i}(y)P_{n-2}^{(i)}(y)\Big].
\eeq*
We note briefly, that $R_2(n,y)\equiv0$ because $\Delta_i^{(k)}(n,y)=0$ for $i=1,\, 2;\ k=0,\,1$ and the  sums over $k\le m-3$ are void. Therefore, since $P_0^{(2)}(y)=P_1^{(2)}(y)=0$ and $P_k^{(0)}(y)=1$, $k\ge0$,
\eqref{P_n} yields
\[P_n^{(2)}(y)=\sum_{j=0}^{n-2}T_{n-j}(y+j).\]
Letting $y=1$ and using  \eqref{t_def} and \eqref{alpha} we see that
\beq*T_{n-j}(j+1)&=&\frac{(nj+2n-j^2-3j-3)^2}{n^2(n-1)^2}+\frac{nj+2n-j^2-3j-3}{n(n-1)}\\&&\quad+2\frac{(j+2)(n-1)\b_{n-j}(j+1)-(j+1)^2\d_{n-j}(j+1)}{n(n-1)}
\eeq*
and after calculating and substituting the expressions for $\b_{n-j}(j+1)$ and $\d_{n-j}(j+1)$ we obtain 
\[\var(C_n)=P_n^{(2)}(1)=\sum_{j=0}^{n-2}T_{n-j}(j+1)=\frac{11n^4-191n^2+360n+180}{180n^2(n-1)},\]
which agrees with the earlier calculation.

We now show inductively that 
\be\lbl{bdds}P_n^{(m)}(y)=O((n+y)^{\floor{m/2}})\quad \mbox{and} \quad R_{m}(n,y)=O((n+y)^{\floor{(m-3)/2}})\ee and 
 that this is uniform over $n\in\mathbb N$ and $y>0$ such that $n+y\ge y_0$.

Notice that the first part of \eqref{bdds}   immediately implies \eqref{mom} for odd moments since if $m=2r+1$ is odd, then 
\beq*
\frac{P_n^{(2r+1)}(1)}{{(\frac{11} {180} n)}^{\frac {2r+1}2}}=\frac{O(n^{r})}{(\frac{11} {180} n)^{r+\frac12}}\longrightarrow 0
\eeq*
as $n\rightarrow \infty$.

To prove \eqref{bdds} we first observe that  $P_n^{(0)}(y)=1$, $P_n^{(1)}(y)=0$ and $P_n^{(2)}(y)=O(n+y)$. Assume now that $P_n^{(r)}(y)=O((n+y)^{\floor{r/2}})$ for $r=0,1,...,m$. In order to prove that $P_n^{(m+1)}(y)=O((n+y)^{\floor{\frac{m+1}2}})$, we first 
 show that the differences $\Delta_i^{(m)}(n,y)$,  $i=1,\, 2$,
lose an order of magnitude, that is 
if $P_n^{(m)}(y)=O((n+y)^{\floor{m/2}})$, then  $\Delta_i^{(m)}(n,y)=
O((n+y)^{\floor{\frac{m-2}2}})$, $i=1,\,2$. 
From \eqref{mgf}, we have that 
\beq*
P_n^{(m)}(y)=
\frac{\partial^m}{\partial t^m}\left(e^{-\mu_n(t)y}C_n(e^t,y)\right)\Big|_{t=0}
=\frac1{y^{\ris{n}}}\sum_{k=0}^m \binom mk (-\mu_n(y))^k C_n^{(m-k)}(1,y).
\eeq*

We note that $C_n^{(m-k)}(1,y)$ is a polynomial in $n$ and $y$, and so $P_n^{(m)}(y)$ is a rational function, say, 
 \[P_n^{(m)}(y)=\frac{p_m(n,y)}{q_m(n,y)}\] where $p_m(n,y)$ and $q_m(n,y)$ are polynomials.  
 The asymptotics of  
 $P_n^{(m)}(y)$ as $n+y\to\infty$ for $n\in\mathbb N$, $y>0$ are driven by $\deg(p_m)-\deg(q_m)$ where $\deg(p)$ stands for the total degree of a polynomial $p(n,y)$.

Consider 
\beq*
\Delta_2^{(m)}(n,y)&=&
\frac{p_m(n-2,y)q_m(n-1,y+1)-p_m(n-1,y+1)q_m(n-2,y)}{q_m(n-2,y)q_m(n-1,y+1)}.
\eeq*
By expanding the powers of $(n-1)^k=(n-2+1)^k$ and $(y+1)^l$ in $q_m(n-1,y+1)$ and $p_m(n-1,y+1)$ using the binomial formula,  
we see that all monomials of the highest total degree in the numerator cancel.  Thus, the polynomial in the numerator has degree at most 
$\deg(p_m)+\deg(q_m)-1$ and since the degree of the denominator is $2\deg(q_m)$, the growth rate of $\Delta_2^{(m)}(n,y)$ is at most 
$\deg(p_m)-\deg(q_m)-1$, one order of magnitude less than the growth rate of $P_n^{(m)}(y)$.
The argument for $\Delta_1^{(m)}(n,y)$ is the same.  

It thus follows from the inductive hypothesis that $\Delta_i^{(m)}(n,y)= O((n+y)^{\floor{\frac{m-2}2}})$, for $i=1,2$ and that $R_{m+1}(n,y)=O((n+y)^{\floor{\frac{m-2}2}})$ (for the latter fact, we use that 
\[\left|\frac{(y+1)(y+n-2)}{(y+n-1)_2}(1+2\beta_n(y))\right|\le C,\quad \left|\frac{y^2}{(y+n-1)_2}(1+2\d_n(y))\right|\le C\]
for a  universal constant $C$ all $n\in\mathbb N$ and $y>0$ such that $n+y\ge y_0$, which follows from \eqref{coeff_bdds}).

Then we see that
\beq*
 P_n^{(m+1)}(y)&=&\binom{m+1}2\sum_{j=1}^{n-1}\Big[T_{n-j+1}(y+j-1)P_{n-j}^{(m-1)}(y+j)+O(((n+y)^{\floor{\frac{m-2}2}})\Big]
\\&=&\binom{m+1}2\sum_{j=1}^{n-1}\Big[O(1)O((n+y)^{\floor{\frac{m-1}2}})+O(((n+y)^{\floor{\frac{m-2}2}})\Big]
\\&=&\binom {m+1}2 \Big[O((n+y)^{\floor{\frac{m-1}2}+1})+O(((n+y)^{\floor{\frac{m-2}2}+1})\Big]
\\&=&O((n+y)^{\floor{\frac{m+1}2}})
\eeq*
which concludes the induction. 

We can now complete the proof of \eqref{mom} when $m=2r$ is even. For $n$ sufficiently large, repeated application of \eqref{P_n} yields 
\beq*P_n^{(2r)}(y)
&=&{2r\choose2}\sum_{j=1}^{n-2}[T_{n-j+1}(y+j-1)P_{n-j}^{(2r-2)}(y+j)+O((n+y)^{r-2})]\\
&=&{2r\choose2}{{2r-2}\choose2} \sum_{j_1=1}^{n-2}\sum_{j_2=1}^{n-j_1-2}\Biggr\{T_{n-j_1+1}(y+j_1-1)T_{n-j_1-j_2+1}(y+j_1+j_2-1)\\
&&\hspace{1cm}\times\Big(P_{n-j_1-j_2}^{(2r-4)}(y+j_1+j_2)
+O((n+y)^{r-3})\Big)\Biggr\}+\sum_{j=1}^{n-2}O((n+y)^{r-2})\\
&=&{2r\choose2}{{2r-2}\choose2} \sum_{j_1=1}^{n-2}\sum_{j_2=1}^{n-j_1-2}\Biggr\{T_{n-j_1+1}(y+j_1-1)T_{n-j_1-j_2+1}(y+j_1+j_2-1)\\&&\hspace{4cm}\times P_{n-j_1-j_2}^{(2r-4)}(y+j_1+j_2)\Biggr\}
\\&&\hspace{1.5cm}+O\left(\left(\sum_{j=1}^{n-2}T_{n-j}(y+j)\right)^2\cdot (n+y)^{r-3}\right)+O((n+y)^{r-1})\\
&=&{2r\choose2}{{2r-2}\choose2} \sum_{j_1=1}^{n-2}\sum_{j_2=1}^{n-j_1-2}\Biggr\{T_{n-j_1+1}(y+j_1-1)T_{n-j_1-j_2+1}(y+j_1+j_2-1)\\&&\hspace{4cm}\times P_{n-j_1-j_2}^{(2r-4)}(y+j_1+j_2)\Biggr\}
\\&&\hspace{1.5cm} +O(n^2(n+y)^{r-3})+O((n+y)^{r-1})=\dots=\\
&=&{2r\choose2}
\dots{4\choose2}\sum_{j_1=1}^{n-2}\sum_{j_2=1}^{n-j_1-2}\dots\sum_{j_r=1}^{n-1-\sum_{i=1}^{r-1}j_i}\prod_{i=1}^r T_{n+1- \sum_{l=1}^{i}j_l}(y+\sum_{l=1}^{i}j_l-1)
\\&&\hspace{1.5cm}
+O((n+y)^{r-1})\\
&=&\frac{(2r)!}{2^r}\sum_{0\le k_1<k_2<\dots<k_r<n-1}\prod_{i=1}^r T_{n-k_i}(y+k_i)+O((n+y)^{r-1})\\
&=&\frac{(2r)!}{2^r}\frac1{r!}\sum_{0\le k_1,\dots, k_r<n-1\atop {\rm distinct}}\prod_{i=1}^r T_{n-k_i}(y+k_i)+O((n+y)^{r-1})\\
&=&\frac{(2r)!}{2^r}\frac1{r!}\left(\sum_{0\le k_1,\dots,k_r<n-1\atop {\rm all}} \prod_{i=1}^r T_{n-k_i}(y+k_i)-\sum_{0\le k_1,\dots,k_r<n-1\atop {\rm not\  all\   distinct}} \prod_{i=1}^r T_{n-k_i}(y+k_i)\right)
\\&\quad +&O((n+y)^{r-1}).
\eeq*
Set $y=1$. 
Then the first sum becomes
\[\sum_{0\le k_1,\dots,k_r<n-1\atop {\rm all}} \prod_{i=1}^r T_{n-k_i}(1+k_i)=
\left(\sum_{k=0}^{n-2} T_{n-k}(k+1)\right)^r=\left(P_n^{(2)}(1)\right)^r\sim\left(\frac{11}{180}n\right)^r\]
where the second equality follows from \eqref{P_n} used with $m=2$ and  $P_n^{(0)}(y)=1$. 

For the second summation, recall that by \eqref{t_bdd}, $|T_{n-k}(y+k)|\le C$. Thus:
 \[\left|\sum_{0\le k_1,\dots,k_r<n-1\atop{\rm not\ all\ distinct}} \prod_{i=1}^r T_{n-k_i}(y+k_i)\right|\le \sum_{0\le k_1,\dots,k_r<n-1\atop {\rm not\ all\ distinct}} \prod_{i=1}^r 
 \left|T_{n-k_i}(y+k_i)\right|\le 
C^r\cdot O(n^{r-1}),\]
which is of lower order than the first sum. This proves \eqref{mom} for $m$ even and completes the proof of Theorem~\ref{main_thm}.

\bibliographystyle{plain}

\begin{thebibliography}{10}

\bibitem{abn}
J.-C. Aval, A.~Boussicault, and P.~Nadeau.
\newblock Tree-like tableaux.
\newblock {\em Electron. J. Combin.}, 20(4):Paper 34, 24, 2013.

\bibitem{bil}
P.~Billingsley.
\newblock {\em Probability and measure}.
\newblock Wiley Series in Probability and Mathematical Statistics. John Wiley
  \& Sons, Inc., New York, third edition, 1995.
\newblock A Wiley-Interscience Publication.

\bibitem{bur}
A.~Burstein.
\newblock On some properties of permutation tableaux.
\newblock {\em Ann. Comb.}, 11(3-4):355--368, 2007.

\bibitem{ch}
S.~Corteel and P.~Hitczenko.
\newblock Expected values of statistics on permutation tableaux.
\newblock In {\em 2007 {C}onference on {A}nalysis of {A}lgorithms, {A}of{A}
  07}, Discrete Math. Theor. Comput. Sci. Proc., AH, pages 325--339. Assoc.
  Discrete Math. Theor. Comput. Sci., Nancy, 2007.

\bibitem{cn}
S.~Corteel and P.~Nadeau.
\newblock Bijections for permutation tableaux.
\newblock {\em European J. Combin.}, 30(1):295--310, 2009.

\bibitem{cw}
S.~Corteel and L.~K. Williams.
\newblock Tableaux combinatorics for the asymmetric exclusion process.
\newblock {\em Adv. in Appl. Math.}, 39(3):293--310, 2007.

\bibitem{ds04}
M.~Depken and R.~B. Stinchcombe.
\newblock Exact joint density--current probability function for the asymmetric
  exclusion process.
\newblock {\em Phys. Rev. Lett.}, 93(2):040602, 2004.

\bibitem{ds05}
M.~Depken and R.~B. Stinchcombe.
\newblock Exact probability function for bulk density and current in the
  asymmetric exclusion process.
\newblock {\em Phys. Rev.}, 71(3):036120, 2005.

\bibitem{dehp}
B.~Derrida, M.~R. Evans, V.~Hakim, and V.~Pasquier.
\newblock Exact solution of a {$1$}{D} asymmetric exclusion model using a
  matrix formulation.
\newblock {\em J. Phys. A}, 26(7):1493--1517, 1993.

\bibitem{ggls}
A.~L.~L. Gao, E.~X.~L. Gao, P.~Laborde-Zubieta, and B.~Y. Sun.
\newblock Enumeration of corners in tree-like tableaux.
\newblock {\em Discrete Math. Theor. Comput. Sci.}, 18(3):Paper No. 17, 26,
  2016.

\bibitem{hj}
P.~Hitczenko and S.~Janson.
\newblock Asymptotic normality of statistics on permutation tableaux.
\newblock In {\em Algorithmic probability and combinatorics}, volume 520 of
  {\em Contemp. Math.}, pages 83--104. Amer. Math. Soc., Providence, RI, 2010.

\bibitem{hl}
P.~Hitczenko and A.~Lohss.
\newblock Corners in tree-like tableaux.
\newblock {\em Electron. J. Combin.}, 23(4):Paper 4.26, 18, 2016.

\bibitem{lz}
P.~Laborde-Zubieta.
\newblock Occupied corners in tree-like tableaux.
\newblock {\em S\'em. Lothar. Combin.}, 74:Art. B74b, 14, [2015-2017].

\bibitem{sw}
E.~Steingr{\'\i}msson and L.~K. Williams.
\newblock Permutation tableaux and permutation patterns.
\newblock {\em J. Combin. Theory Ser. A}, 114(2):211--234, 2007.

\bibitem{sq}
R.~B. Stinchcombe and S.~L.~A. de~Quieros.
\newblock Statistics of current activity fluctuations in asymmetric flow with
  exclusion.
\newblock {\em Phys. Rev. E}, 85(4):041111, 2012.

\end{thebibliography}

\end{document}